\documentclass[11pt]{article}

\usepackage[utf8]{inputenc}
\usepackage[T1]{fontenc}             
\usepackage{lmodern}
\usepackage[english]{babel}
\usepackage{amsmath}
\usepackage{amsthm}
\usepackage{amsmath}
\usepackage{amssymb}
\usepackage{mathrsfs}
\usepackage{geometry}
\usepackage{graphicx}
\usepackage{comment}
\usepackage{url}
\usepackage{titlesec}
\usepackage[all]{xy}
\usepackage{enumitem}
\usepackage{amssymb,amsmath,amsfonts,amsxtra,amsthm}
\usepackage[all]{xy}
\usepackage{stmaryrd}
\usepackage[colorlinks=red]{hyperref}

\newtheorem{df}{Definition}[section]
\newtheorem{rem}[df]{Remark}
\newtheorem{ex}[df]{Example}
\newtheorem{thm}[df]{Theorem}
\newtheorem{pp}[df]{Proposition} 
\newtheorem{lm}[df]{Lemma}

\newtheorem{cor}[df]{Corollary}

\newtheorem*{notation}{Notation}

\def\so{\mathfrak{so}}

\def\sl{\mathfrak{sl}}

\def\sp{\mathfrak{sp}}

\def\gg{\mathfrak{g}}
\def\hh{\mathfrak{h}}
\def\ppp{\mathfrak{p}}

\def\ss{\mathfrak{s}}

\def\kt{\tilde{k}}
\def\gt{\tilde{\gg}}

\def\osp{\mathfrak{osp}}

\title{\Large{\textbf{Classification of finite-dimensional Lie superalgebras whose even part is a three-dimensional simple Lie algebra over a field of characteristic not two or three}}}
\author{Philippe Meyer}
\date{\vspace{-5ex}}

\begin{document}
\maketitle
\begin{center}
\textbf{Abstract}
\end{center}
Let $k$ be a field of characteristic not two or three. We classify up to isomorphism all finite-dimensional Lie superalgebras $\mathfrak{g}=\mathfrak{g}_0\oplus \mathfrak{g}_1$ over $k$, where $\mathfrak{g}_0$ is a three-dimensional simple Lie algebra. If $\mathcal{Z}(\mathfrak{g})$ denotes the centre of $\mathfrak{g}$, the result is the following: either $\lbrace \mathfrak{g}_1,\mathfrak{g}_1 \rbrace=\lbrace 0 \rbrace$ or $\mathfrak{g}_1=(\mathfrak{g}_0 \oplus k)\oplus \mathcal{Z}(\mathfrak{g})$ or $\mathfrak{g}\cong \mathfrak{osp}_k(1|2)\oplus \mathcal{Z}(\mathfrak{g})$.
\vspace{0.4cm}

\vspace{0.4cm}


\section{Introduction}
Since the work of E. Cartan, it has been well-known that the structure and representation theory of the smallest simple complex Lie algebra $\sl(2,\mathbb{C})$ are the keys to the classification of all finite-dimensional, simple complex Lie algebras. Lie superalgebras are generalisations of Lie algebras and from this point of view, it is natural to ask what are the Lie superalgebras whose even part is a three-dimensional simple Lie algebra.
\vspace{0.2cm}

Over the complex numbers and with the extra assumption that $\gg$ is simple, the answer to this question can be extracted from the classification of finite-dimensional, simple complex Lie superalgebras by V. Kac (see \cite{Kac1977}). In this case the only possibility, up to isomorphism, is the complex orthosymplectic Lie superalgebra $\osp_{\mathbb{C}}(1|2)$. Over the real numbers and again with the extra assumption that $\gg$ is simple, the answer can similarly be extracted from the classification of finite-dimensional, real simple Lie superalgebras by V. Serganova (see \cite{Serganova1983}) and is the same: the only possibility, up to isomorphism, is the real orthosymplectic Lie superalgebra $\osp_{\mathbb{R}}(1|2)$. However if $k$ is a general field, there is currently no classification of finite-dimensional, simple Lie superalgebras over $k$ to which we can appeal to answer the question above. Nevertheless, let us point out, if $k$ is algebraically closed and if $char(k)>5$, S. Bouarroudj and D. Leites have conjectured a list of all the finite-dimensional, simple Lie superalgebras over $k$ (see \cite{Leit1}) and S. Bouarroudj, P. Grozman and D. Leites have classified finite-dimensional Lie superalgebras over $k$ with indecomposable Cartan matrices under the assumption that they have a Dynkin diagram with only one odd node (see \cite{Leit2}). In both cases, the only Lie superalgebra whose even part is a three-dimensional simple Lie algebra which appears is $\osp_k(1|2)$.
\vspace{0.2cm}

In this paper we will give a classification of all finite-dimensional Lie superalgebras $\gg$ over a field $k$ of characteristic not two or three whose even part is a three-dimensional simple Lie algebra. For our classification, we do not assume that $k$ is algebraically closed and we do not assume that $\gg$ is simple. The main result (Theorem \ref{thm final}) is:
\vspace{0.3cm}

\noindent
\textbf{Theorem.}
\textit{Let $k$ be a field of characteristic not two or three. Let $\gg=\gg_0 \oplus \gg_1$ be a finite-dimensional Lie superalgebra over $k$ such that $\gg_0$ is a three-dimensional simple Lie algebra and let
$$\mathcal{Z}(\gg):= \lbrace x \in \gg ~ | ~ \lbrace x,y \rbrace=0 \quad \forall y \in \gg \rbrace .$$
Then, there are three cases:}
\begin{enumerate}[label=\alph*)] \it
\item $\lbrace \gg_1,\gg_1 \rbrace=\lbrace 0 \rbrace$ ;
\item $\gg_1=(\gg_0\oplus k) \oplus \mathcal{Z}(\gg)$, \quad (see Example \ref{ex ss+ss+k}) ;
\item $\gg\cong \osp_k(1|2)\oplus \mathcal{Z}(\gg)$, \quad (see Example \ref{osp(12)}).
\end{enumerate}
\vspace{0.2cm}

In $b)$, the non-trivial brackets on $\gg_1=(\gg_0\oplus k)\oplus \mathcal{Z}(\gg)$ are given by
$$\lbrace v,\lambda \rbrace=\lambda v \quad \forall v \in \gg_0, ~ \forall \lambda \in k$$
where the right-hand side of this equation is to be understood as an even element of $\gg$.
\vspace{0.2cm}

It follows from this classification that $\gg$ is simple if and only if $\gg\cong \osp_k(1|2)$ or $\gg_1=\lbrace 0 \rbrace$. It also follows that, if $k$ is of positive characteristic $p$ and the restriction of the bracket to $\gg_1$ is non-zero, then $\gg$ is a restricted Lie superalgebra in the sense of V. Petrogradski (\cite{PETROGRADSKI19921}) and Y. Wang-Y. Zhang (\cite{Wang2000}). See Corollary \ref{cor restricted} for explicit formulae for the $[p]|[2p]$-mapping.
\vspace{0.2cm}

We first prove the theorem when $\gg_0$ is $\sl(2,k)$ and $k$ is either of characteristic zero or of positive characteristic and algebraically closed. An essential point here is that for such fields, the classification of finite-dimensional, irreducible representations of $\sl(2,k)$ is known (for $k$ of positive characteristic and algebraically closed see for example \cite{RudShaf} or \cite{StradeF}). This, together with a careful study of the restriction of the bracket to irreducible $\sl(2,k)$-submodules of $\gg_1$, is the main ingredient of the proof. It turns out that the main difficulty occurs when $k$ is of positive characteristic and the only irreducible submodules of $\gg_1$ are trivial. In this case we do not have complete reducibility of finite-dimensional representations of $\sl(2,k)$ but nevertheless, using notably an observation of H. Strade (see \cite{Strade}), we show that the bracket restricted to $\gg_1$ is trivial as expected.
\vspace{0.2cm}

Once we have proved our result under the restricted hypotheses above, we use three rather general results (see Propositions \ref{non-split 2-dim rep}, \ref{supercentre} and \ref{bourbaki}) to extend it to the case when $k$ is not algebraically closed and $\gg_0$ is not necessarily isomorphic to $\sl(2,k)$. Recall that if $k$ is not algebraically closed there are in general many three-dimensional simple Lie algebras over $k$, not just $\sl(2,k)$ (see \cite{MALCOLMSON1992}).
\vspace{0.5cm}

The paper is organised as follows. In Section 2 we give a precise definition of the Lie superalgebras which appear in our classification. In Section 3 we give some general consequences of the Jacobi identities of a Lie superalgebra whose even part is simple. In Section 4 we recall what is known on the structure of finite-dimensional irreducible representations of $\sl(2,k)$ and a criterion of complete reducibility in positive characteristic due to N. Jacobson (see \cite{Jacobson1958}). In Section 5, assuming that $k$ is of characteristic zero or of positive characteristic and algebraically closed, we prove some vanishing properties of the bracket of a Lie superalgebra of the form $\gg=\sl(2,k)\oplus \gg_1$. In the last section we prove the main results of the paper (Theorems \ref{classification LSA alg clos} and \ref{thm final}) and give some counter-examples in characteristic two and three.
\vspace{0.5cm}

\textit{Throughout this paper, the field $k$ is always of characteristic not two or three (except in the comments and examples given after Corollary \ref{cor restricted}).}

\section{Examples of Lie superalgebras}

In a $\mathbb{Z}_2$-graded vector space $\gg=\gg_0\oplus \gg_1$, the elements of $\gg_0$ are called even and those of $\gg_1$ are called odd. We denote by $|v|\in \mathbb{Z}_2$ the parity of a homogeneous element $v\in \gg$ and whenever this notation is used, it is understood that $v$ is homogeneous.

\begin{df} \label{def LSA}
A Lie superalgebra is a $\mathbb{Z}_2$-graded vector space $\gg=\gg_0\oplus \gg_1$ together with a bilinear map $\lbrace \phantom{x}, \phantom{x} \rbrace : \gg \times \gg \rightarrow \gg $ such that 
\begin{enumerate}[label=\alph*)]
\item $\lbrace \gg_{\alpha}, \gg_{\beta} \rbrace \subseteq \gg_{\alpha+\beta}$ for all $\alpha,\beta \in \mathbb{Z}_2$,
\item $\lbrace x,y \rbrace=-(-1)^{|x||y|}\lbrace y,x \rbrace$ for all $x,y \in \gg$,
\item $(-1)^{|x||z|}\lbrace x , \lbrace y , z \rbrace \rbrace+(-1)^{|y||x|}\lbrace y , \lbrace z , x \rbrace \rbrace+(-1)^{|z||y|}\lbrace z , \lbrace x , y \rbrace \rbrace=0$ for all $x,y,z \in \gg$.
\end{enumerate}
\end{df}

\begin{ex} \label{exEnd(V)}
Let $V=V_0\oplus V_1$ be a $\mathbb{Z}_2$-graded vector space. The algebra {\rm End(V)} is an associative superalgebra for the $\mathbb{Z}_2$-gradation defined by
$${\rm End(V)}_a:= \lbrace f \in {\rm End(V)} ~ | ~ f(V_b) \subseteq  V_{a+b} \quad \forall b \in \mathbb{Z}_2 \rbrace$$
for all $a\in \mathbb{Z}_2$. Now, if we define for all $v,w \in {\rm End(V)}$
$$\lbrace v,w \rbrace:=vw-(-1)^{|v||w|}wv,$$
then ${\rm End(V)}$ with the $\mathbb{Z}_2$-gradation defined above and the bracket $\lbrace \phantom{v}, \phantom{v} \rbrace$ is a Lie superalgebra.
\end{ex}

A Lie superalgebra can also be thought as a Lie algebra and a representation carrying an extra structure.

\begin{pp} Let $\gg_0 \rightarrow {\rm End(\gg_1)}$ be a representation of a Lie algebra $\gg_0$ and let $P:\gg_1 \times \gg_1 \rightarrow \gg_0$ be a symmetric bilinear map.
\vspace{0.2cm}

The vector space $\gg:=\gg_0\oplus \gg_1$ with the bracket $\lbrace \phantom{x}, \phantom{x} \rbrace : \gg \times \gg \rightarrow \gg $ defined by
\begin{enumerate}[label=\alph*)]
\item $\lbrace x,y \rbrace :=[x,y]$ for $x,y \in \gg_0$,
\item $\lbrace x,v \rbrace:=-\lbrace v,x \rbrace :=x(v)$ for $x \in \gg_0$ and $v \in \gg_1$,
\item $\lbrace v,w \rbrace :=P(v,w)$ for $v,w \in \gg_1$,
\end{enumerate}
is a Lie superalgebra if and only if the map $P$ satisfies the two relations:
\begin{equation}
[x,P(u,v)]=P(x(u),v)+P(u,x(v)) \qquad \forall x \in \gg_0, ~ \forall u,v \in \gg_1,
\label{relations P 2}
\end{equation}
\begin{equation}
P(u,v)(w)+P(v,w)(u)+P(w,u)(v)=0 \qquad \forall u,v,w \in \gg_1.
\label{relations P 1}
\end{equation}
\end{pp}

\begin{proof}
Straightforward.
\end{proof}
\vspace{0.2cm}

\begin{rem}
Given a Lie algebra $\gg_0$, one can ask which representations $\gg_0 \rightarrow {\rm End(\gg_1)}$ arise as the odd part of a non-trivial Lie superalgebra $\gg=\gg_0\oplus \gg_1$. In \cite{Kos}, this question is considered in the case of symplectic complex representations of quadratic complex Lie algebras (i.e. those admitting a non-degenerate symmetric ad-invariant bilinear form). 
\vspace{0.2cm}

Let $\rho : \gg_0 \rightarrow \sp(\gg_1)$ be a finite-dimensional, symplectic complex representation of a finite-dimensional, quadratic complex Lie algebra. 
\vspace{0.2cm}

B. Kostant shows that there exists a Lie superalgebra structure on $\gg:=\gg_0\oplus\gg_1$ compatible with the natural super-symmetric bilinear form, if and only if the image of the invariant quadratic form on $\gg_0$ under the induced map from the envelopping algebra of $\gg_0$ to the Weyl algebra of $\gg_1$ satisfies a certain identity. In this case the map $P$ is uniquely determined.
\label{kostant}
\end{rem}
\vspace{0.2cm}

Let $\gg$ be a Lie algebra. Using the adjoint representation of $\gg$ and a doubling process, we can construct a Lie superalgebra $\gg_0 \oplus \gg_1$ such that $\gg_0$ is isomorphic to $\gg$ and such that $P$ is non-trivial. 

\begin{df} \label{def dbl}
Let $\gg$ and $\gg'$ be isomorphic Lie algebras and let $\phi : \gg \rightarrow \gg'$ be an isomorphism of Lie algebras. Let 
$$\gt:=\gg \oplus (\gg'\oplus \mathcal{Z}_{\gg}(\gg')),$$
where 
$$\mathcal{Z}_{\gg}(\gg'):=\lbrace f \in Hom( \gg , \gg' ) ~ | ~ f \circ ad(x)=ad(\phi(x)) \circ f \quad  \forall x \in \gg \rbrace.$$
We define a $\mathbb{Z}_2$-gradation of $\gt$ by
$$\gt_0:=\gg, \qquad \gt_1:=\gg'\oplus \mathcal{Z}_{\gg}(\gg')$$
and a $\mathbb{Z}_2$-graded skew-symmetric bilinear bracket $\lbrace \phantom{v},\phantom{v} \rbrace$ on $\gt$ by:
\begin{itemize}
\item $\lbrace x,y \rbrace:=[x,y] \quad $ for $x,y \in \gg$ ;
\item $\lbrace x,v \rbrace:=[\phi(x),v] \quad$ for $x \in \gg, ~ v \in \gg'$ ;
\item $\lbrace x,f \rbrace:=0 \quad$ for $x \in \gg, ~ f \in \mathcal{Z}_{\gg}(\gg')$ ;
\item $\lbrace v,w \rbrace:=\lbrace f,g \rbrace:=0 \quad$ for $v,w \in \gg', ~ f,g \in \mathcal{Z}_{\gg}(\gg')$ ;
\item $\lbrace v,f \rbrace:=\phi^{-1}(f(\phi^{-1}(v))) \quad$ for $v \in \gg', ~ f \in \mathcal{Z}_{\gg}(\gg')$. 
\end{itemize}
\end{df}
\vspace{0.1cm}

\begin{rem}
The Lie algebra $\gg'$ is isomorphic to $\gg$ and so $\mathcal{Z}_{\gg}(\gg')\cong \mathcal{Z}_{\gg}(\gg)$.
\end{rem}
\vspace{0.1cm}

\begin{pp}
The vector space $\gt:=\gg \oplus (\gg'\oplus \mathcal{Z}_{\gg}(\gg'))$ together with the $\mathbb{Z}_2$-gradation and bracket $\lbrace \phantom{v},\phantom{v} \rbrace$ above is a Lie superalgebra.
\end{pp}

\begin{proof}
We have to check the two relations \eqref{relations P 2} and \eqref{relations P 1}. Let $x \in \gg$, $v \in \gg'$ and $f \in \mathcal{Z}_{\gg}(\gg')$, we have
$$\begin{array}{cl}
 [x,P(v,f)]&=[x,\phi^{-1}(f(\phi^{-1}(v)))]=\phi^{-1}([\phi(x),f(\phi^{-1}(v))])=\phi^{-1}(f([x,\phi^{-1}(v)]))  \\
&=\phi^{-1}(f(\phi^{-1}([\phi(x),v])))=P([\phi(x),v],f)=P(\lbrace x,v \rbrace,f)+P(v, \lbrace x,f \rbrace),
\end{array}$$
and so the relation \eqref{relations P 2} is satisfied. We only need to check the relation \eqref{relations P 1} for $f \in \mathcal{Z}_{\gg}(\gg')$ and $v,w \in \gg'$:
\begin{align*}
\lbrace \lbrace f,v \rbrace ,w \rbrace  +\lbrace \lbrace v,w \rbrace ,f \rbrace +\lbrace \lbrace f,w \rbrace ,v \rbrace &= \lbrace \phi^{-1}(f(\phi^{-1}(v))),w \rbrace + \lbrace \phi^{-1}(f(\phi^{-1}(w))),v \rbrace\\
&=[f(\phi^{-1}(v)),w]+[f(\phi^{-1}(w)),v]\\
&=f([\phi^{-1}(v),\phi^{-1}(w)])+f([\phi^{-1}(w),\phi^{-1}(v)])\\
&=0.
\end{align*}
\end{proof}

\begin{rem}
\begin{enumerate}[label=\alph*)]
\item This Lie superalgebra cannot be obtained by Kostant's construction since $k\oplus \gg'$ is not a symplectic representation of $\gg$.
\item This Lie superalgebra is not simple.
\end{enumerate}
\end{rem}
\vspace{0.2cm}

From the point of view of this paper, the most interesting case of this construction is when $\gg$ is a three-dimensional simple Lie algebra.

\begin{ex}
Let $\ss$ be a three-dimensional simple Lie algebra over $k$ and let $\bar{k}$ be the algebraic closure of $k$. We have
$$\mathcal{Z}_{\ss}(\ss)\otimes \bar{k} \cong \mathcal{Z}_{\ss\otimes \bar{k}}(\ss\otimes \bar{k}).$$
Since $\ss\otimes \bar{k}\cong \sl(2,\bar{k})$ is simple, by Schur's Lemma we obtain that $\mathcal{Z}_{\ss\otimes \bar{k}}(\ss\otimes \bar{k})\cong \bar{k}$ and hence $\mathcal{Z}_{\ss}(\ss)\cong k$. 
\vspace{0.2cm}

In this case, the Lie superalgebra defined above is isomorphic to
$$\ss\oplus (\ss\oplus k).$$

When $\ss$ is split, the Lie superalgebra $\ss\oplus (\ss\oplus k)$ is isomorphic to the ``strange'' Lie superalgebra $\ppp(1)$ (see section 2.4 in \cite{Musson}).

\label{ex ss+ss+k}
\end{ex}
\vspace{0.2cm}

We now introduce the other type of Lie superalgebra which we will need later on in the paper. These are the orthosymplectic Lie superalgebras whose definition and properties we now recall (for more details see \cite{SCH}).
\vspace{0.2cm}

\begin{df}
Let $V=V_0\oplus V_1$ be a finite-dimensional $\mathbb{Z}_2$-graded vector space together with a non-degenerate even supersymmetric bilinear form $B$, i.e., $(V_0,B|_{V_0})$ is non-degenerate quadratic vector space, $(V_1,B|_{V_1})$ is a symplectic vector space and $V_0$ is $B$-orthogonal to $V_1$. We define the orthosymplectic Lie superalgebra to be the vector space $\osp_k(V,B):=\osp_k(V,B)_0\oplus\osp_k(V,B)_1$ where
$$\osp_k(V,B)_i:=\lbrace f\in End(V)_i ~ | ~ B(f(v),v')+(-1)^{|f||v|}B(v,f(v'))=0 \quad \forall v,v' \in V\rbrace \quad \forall i \in \mathbb{Z}_2.$$
\end{df}
\vspace{0.2cm}

We can check that $\osp_k(V,B)$ is closed under the bracket defined in Example \ref{exEnd(V)} and is in fact a simple Lie subsuperalgebra of ${\rm End(V)}$ with
$$\osp_k(V,B)_0 \cong \so(V_0,B|_{V_0}) \oplus \sp(V_1,B|_{V_1}), \qquad \osp_k(V,B)_1 \cong V_0 \otimes V_1.$$

\begin{rem}
The Lie superalgebra $\osp_k(V,B)$ can also be obtained from a symplectic representation of a quadratic Lie algebra as follows (cf. Remark \ref{kostant}).
\vspace{0.2cm}

Since $V_0$ is quadratic and $V_1$ is symplectic, there is a natural symplectic form $B|_{V_0}\otimes B|_{V_1}$ on $V_0 \otimes V_1$ and so the representation $\so(V_0,B|_{V_0})\oplus \sp(V_1,B|_{V_1}) \rightarrow {\rm End(V_0 \otimes V_1)}$ is a symplectic representation of a quadratic Lie algebra.
\end{rem}

The orthosymplectic Lie superalgebra which is relevant in this paper is the following.

\begin{ex} \label{osp(12)}
Let 
$$\osp_k(1|2):= \sl(2,k)\oplus k^2$$
be the Lie superalgebra defined by the standard representation $k^2$ of $\sl(2,k)$ and by the moment map $P : k^2 \times k^2 \rightarrow \sl(2,k) $ given by
$$P\left( (a,b),(c,d) \right)=\begin{pmatrix}
-(ad+bc) & 2ac  \\
-2bd & ad+bc 
\end{pmatrix}.$$
\vspace{0.2cm}

If $V=V_0\oplus V_1$ is a $\mathbb{Z}_2$-graded vector space together with a non-degenerate even super-symmetric bilinear form $B$ such that $V_0$ is one-dimensional and $V_1$ two-dimensional, it is easy to see that
$$\osp_k(V,B)\cong \osp_k(1|2).$$
\end{ex}

\section{Generalities on Lie superalgebras with simple even part}

In this section we investigate some of the consequences of the identities \eqref{relations P 2} and \eqref{relations P 1} when the even part of the Lie superalgebra is simple.
\vspace{0.2cm}

\begin{lm} \label{LM general}
Let $V$ be a representation of a simple Lie algebra $\gg$, let $P:V\times V \rightarrow \gg$ be a symmetric bilinear map and let $W \subseteq V$ be a $\gg$-submodule.

\begin{enumerate}[label=\alph*)]
\item \label{LM P_VxW=0 implies P=0} Suppose that the map $P$ satisfies the relation \eqref{relations P 1} and $\gg$ acts non-trivially on W. If
$$P(V,W)=\lbrace 0 \rbrace$$
then we have $P\equiv 0$.
\item \label{non nul implique x(v) dans W} Suppose that $P$ satisfies the relations \eqref{relations P 2} and \eqref{relations P 1} and that $P(W,W)\neq \lbrace 0 \rbrace$. Then
$$\gg \cdot V \subseteq W.$$
\item \label{LM structure subrep triviale} Suppose that the map $P$ satisfies the relation \eqref{relations P 2} and $\gg$ acts trivially on W. Then we have $P(W,W)=\lbrace 0 \rbrace$.
\end{enumerate}
\end{lm}

\begin{proof}
\begin{enumerate}[label=\alph*)]
\item Let $v,v' \in V$. By the relation \eqref{relations P 1}, we obtain
$$P(v,v')(w)+P(v,w)(v')+P(v',w)(v)=0 \quad \forall w \in W.$$
By hypothesis $P(v,w)=P(v',w)=0$ and so $P(v,v')(w)=0$ for all $w \in W$. The non-trivial representation W is faithful since $\gg$ is simple and so
$P(v,v')=0$.

\item By the identity \eqref{relations P 2}, ${\rm Span}<P(W,W)>$ is an ideal of $\gg$. Since it is non-trivial by assumption, we have ${\rm Span}<P(W,W)>=\gg$. Hence, if $x \in \gg$ we have
$$x=\sum_i P(w_i,w_i'),$$
for some $w_1,w_1',\dots,w_n,w_n'$ in W. Let $v \in V$. Using the relation \eqref{relations P 1}, this implies
$$x(v)=\sum_i P(w_i,w_i')(v)=-\sum_i (P(w_i,v)(w_i')+P(w_i',v)(w_i))$$
and hence we observe that $x(v) \in W$.

\item Let $w,w' \in W$. Using the relation \eqref{relations P 2} we have
$$[x,P(w,w')]=P(x(w),w')+P(w,x(w'))=0 \quad \forall x \in \gg $$
and so $P(w,w')=0$ because $\gg$ is simple.
\end{enumerate}
\end{proof}

\section{Reducibility of representations of $\sl(2,k)$}

Representation theory of Lie algebras in positive characteristic is quite different from representation theory of Lie algebras in characteristic zero. For example, in positive characteristic, a Lie algebra $\gg$ has no infinite-dimensional irreducible representations and the dimension of the irreducible representations of $\gg$ is bounded. For more details we refer to the survey \cite{Carsten}.

\begin{notation} \begin{enumerate}[label=\alph*)]
\item An $\sl(2,k)$-triple $\lbrace E,H,F \rbrace$ is a basis of $\sl(2,k)$ satisfying:
\begin{equation*}
[E,F]=H, \quad [H,E]=2E, \quad [H,F]=-2F.
\end{equation*}
\item If k is of positive characteristic $p$, we denote by Gk(p) (resp. $[a]$) the image of $\mathbb{Z}$ (resp. a) under the natural map $\mathbb{Z} \rightarrow k$. We will refer to elements of the image of this map as integers in $k$.
\end{enumerate}
\end{notation}

We first recall the structure of the irreducible representations of $\sl(2,k)$ over a field of characteristic zero or an algebraically closed field of positive characteristic (see for example \cite{StradeF} (p. 207-208)). If the field $k$ is of positive characteristic but not algebraically closed, there is no classification of the irreducible representations of $\sl(2,k)$ to the best of the author's knowledge.

\begin{thm} \label{rep irreds sl(2,k) STRADE}
Let $k$ be a field of characteristic zero or an algebraically closed field of positive characteristic. Let $W$ be a finite-dimensional irreducible representation of $\sl(2,k)$.
\begin{enumerate}[label=\alph*)]
\item \label{rep irreds sl(2,k) STRADE a} There exist $\alpha, \beta \in k$, a basis $\lbrace e_0,\ldots,e_m \rbrace$ of $W$ and an $\sl(2,k)$-triple $\lbrace E,H,F \rbrace$ such that:
\begin{align*}
&H(e_i)=(\alpha-2[i])e_i ~ ; \\
&E(e_0)=0 ~ ; &E(e_i)=&[i](\alpha-([i]-1))e_{i-1}, \quad 1 \leq i \leq m ~ ; \\
&F(e_m)=\beta e_0 ~ ; &F(e_i)=&e_{i+1}, \quad 0 \leq i \leq m-1.
\end{align*}
\item \label{rep irreds sl(2,k) STRADE b} If $char(k)=0$, then $\alpha=dim(W)-1$ and $\beta=0$.

\item \label{rep irreds sl(2,k) STRADE c} If $char(k)=p>0$, then $dim(W)\leq p$. If $dim(W)=p$ and $\alpha \in Gk(p)$, then $\alpha=[dim(W)-1]$. If $dim(W)<p$, then $\alpha=[dim(W)-1]$, $\beta=0$.
\end{enumerate}
\end{thm}

\begin{rem}
In particular, we remark that
$$Ann_{\sl(2,k)}(e_i):=\lbrace x \in \sl(2,k) ~ | ~ x(e_i)=0 \rbrace=\left\{
\begin{array}{cl}
\lbrace 0 \rbrace & \text{if } ~ 1\le i \le m-1, ~ i \neq \frac{m}{2} ~ ; \\
{\rm Span}<H> & \text{if } ~ i = \frac{m}{2} ~ ; \\
{\rm Span}<E> & \text{if } ~ i=0 ~ ; \\
{\rm Span}<F> & \text{if } ~ i=m \text{ and } \beta=0 ~ ; \\
\lbrace 0 \rbrace & \text{if }  i=m \text{ and } \beta \ne 0.
\end{array}
\right.$$
In other words, $Ann_{\sl(2,k)}(e_i)$ can only be non-trivial for special values of $i$.
\label{ann}
\end{rem}
\vspace{0.2cm}

We now turn to the question of when a finite-dimensional representation of $\sl(2,k)$ is completely reducible. The following theorem gives sufficient conditions for complete reducibility even if $k$ is not algebraically closed.

\begin{thm} \label{comp red sl(2)} Let $k$ be an arbitrary field. Let $\rho : \sl(2,k) \rightarrow {\rm End(V)}$ be a finite-dimensional representation and let $\lbrace E,H,F \rbrace$ be an $\sl(2,k)$-triple.
\begin{enumerate}[label=\alph*)]
\item If $char(k)=0$, then $V$ is completely reducible.
\item If $char(k)=p>0$ and $\rho(E)^{p-1}=\rho(F)^{p-1}=0$, then $V$ is completely reducible.
\end{enumerate}
\end{thm}

\begin{proof}
The first part is well-known and follows from the Weyl's theorem on complete reducibility. For the second part see \cite{Jacobson1958}.
\end{proof}

\begin{rem}
For other conditions implying complete reducibility of representations of $\sl(2,k)$ over an algebraically closed field of positive characteristic, see \cite{Strade} (p. 252-253).
\end{rem}

\section{Vanishing properties of the bracket restricted to the odd part of a Lie superalgebra whose even part is $\sl(2,k)$}

In this section, we prove two preliminary results which are crucial to the proof of our main theorems. Let $k$ be a field of characteristic zero or an algebraically closed field of positive characteristic. The first result shows that a Lie superalgebra whose even part is $\sl(2,k)$ and whose odd part is an irreducible representation can only be non-trivial if the odd part is two-dimensional. The second result shows that if the restriction of the bracket to a non-trivial irreducible submodule $W$ of the odd part vanishes, then the bracket vanishes identically unless W is three-dimensional.

\begin{pp} \label{1er LM}
Let $k$ be a field of characteristic zero or an algebraically closed field of positive characteristic. Let $W$ be a finite-dimensional irreducible representation of $\sl(2,k)$ and let $P:W\times W \rightarrow \sl(2,k)$ be a symmetric bilinear map which satisfies the relations \eqref{relations P 2} and \eqref{relations P 1}. 
\begin{enumerate}[label=\alph*)]
\item If $dim(W) \neq 2$, we have
$$P\equiv 0.$$
\item \label{1er LM b} If $dim(W)=2$, let $\lbrace e_0,e_1 \rbrace$ be a basis of $W$ and let $\lbrace E,H,F \rbrace$ be an $\sl(2,k)$-triple as in Theorem \ref{rep irreds sl(2,k) STRADE} a). Then, there exists $\gamma \in k$ such that
\begin{equation*}
P(e_0,e_0)=-2\gamma E, \quad P(e_0,e_1)=\gamma H, \quad P(e_1,e_1)=2\gamma F.
\end{equation*}
\end{enumerate}
\end{pp}

\begin{proof}
If $dim(W)=1$, it follows from Lemma \ref{LM general} \ref{LM structure subrep triviale} that $P\equiv 0$, so suppose $dim(W)\ge 2$ from now on. Let $m:=dim(W)-1$, let $\alpha,\beta \in k$, let $\lbrace e_0,\ldots,e_m \rbrace$ be a basis of $W$ and let $\lbrace E,H,F \rbrace$ be an $\sl(2,k)$-triple as in Theorem \ref{rep irreds sl(2,k) STRADE} \ref{rep irreds sl(2,k) STRADE a}.
\vspace{0.3cm}

Suppose that $\alpha$ is not an integer in $k$. This is only possible if $k$ is of positive characteristic, say $p$. We then know from Theorem \ref{rep irreds sl(2,k) STRADE} that 
$$dim(W)=p\neq 2$$
and since
$$[H,P(e_i,e_j)]=2(\alpha-[i]-[j])P(e_i,e_j) \quad \forall ~ 0\leq i,j\leq m$$
it follows that $P(e_i,e_j)$ is either zero or an eigenvector of $H$ corresponding to the eigenvalue $2(\alpha-[i]-[j])$.
The map $ad(H)$ is diagonalisable with eigenvalues which are integers. Since $(\alpha-[i]-[j])$ cannot be an integer by assumption, we have
$$P(e_i,e_j)=0 \quad \forall ~ 0\leq i,j\leq m.$$
This proves a) if $\alpha$ is not an integer in $k$.
\vspace{0.3cm}

Suppose now that $\alpha$ is an integer in $k$. By Theorem \ref{rep irreds sl(2,k) STRADE} \ref{rep irreds sl(2,k) STRADE b} and \ref{rep irreds sl(2,k) STRADE c}, we have $\alpha=[m]$ (with $1\leq m\leq p-1$ if $char(k)=p$), and so
$$[H,P(e_i,e_i)]=2([m]-2[i])P(e_i,e_i) \qquad \forall ~ 0\leq i\leq m.$$
The map $ad(H)$ is diagonalisable with eigenvalues $-2,0,2$ and so if $P(e_i,e_i)\neq 0$ we must have $[m]-2[i]$ equals respectively $-1, 0$ or $1$ in $k$ and $P(e_i,e_i)$ is proportional to respectively $F$, $H$ or $E$.
\\

Suppose $[m]-2[i]=0$. If $char(k)=0$ then $m$ is even and $i=\frac{m}{2}$. If $char(k)=p$, since the conditions $1 \leq m \leq p-1$ and $0 \leq i \leq m$ imply $-p < m-2i < p$, it follows, again, that $m$ is even and $i=\frac{m}{2}$.
\vspace{0.2cm}

Suppose $[m]-2[i]=1$. If $char(k)=0$ then $m$ is odd and $i=\frac{m-1}{2}$. If $char(k)=p$, then either $m$ is odd and $i=\frac{m-1}{2}$ or $m=p-1$ and $i=m$. 
\vspace{0.2cm}

However, in the second case we would have $P(e_m,e_m) \in {\rm Span}<E>$. Since $char(k)\neq 3$ we have (by Equation \eqref{relations P 1})
$$P(e_m,e_m)(e_m)=0$$
and hence $P(e_m,e_m)=0$ by Remark \ref{ann}. Similarly, the relation $[m]-2[i]=-1$ implies that $m$ is odd and $i=\frac{m+1}{2}$.
\\

In conclusion, if $m$ is even we have:
\begin{align}
P(e_{\frac{m}{2}},e_{\frac{m}{2}})&=a H, \nonumber \\
\label{P ei ei =0} P(e_i,e_i)&=0 \quad \forall i \in \llbracket 0,m \rrbracket \backslash \lbrace \frac{m}{2} \rbrace
\end{align}
for some $a \in k$ and if $m$ is odd we have:
\begin{align}
P(e_{\frac{m-1}{2}},e_{\frac{m-1}{2}})&=b E,\nonumber \\
P(e_{\frac{m+1}{2}},e_{\frac{m+1}{2}})&=c F,\nonumber \\
\label{P ei ei =0 2} P(e_i,e_i)&=0 \quad \forall i \in \llbracket0,m \rrbracket \backslash \lbrace \frac{m-1}{2},\frac{m+1}{2} \rbrace
\end{align}
for some $b,c \in k$. We now show that $a=b=c=0$.
\vspace{0.3cm}

If $m$ is even, then $\frac{m}{2}\ge 1$ and by \eqref{P ei ei =0} we have
$$P(e_{\frac{m}{2}-1},e_{\frac{m}{2}-1})=0$$
which by \eqref{relations P 2} means
$$[F,P(e_{\frac{m}{2}-1},e_{\frac{m}{2}-1})]=2P(e_{\frac{m}{2}},e_{\frac{m}{2}-1})=0.$$
Furthermore, by \eqref{relations P 1} we have
$$P(e_{\frac{m}{2}},e_{\frac{m}{2}})(e_{\frac{m}{2}-1})+2P(e_{\frac{m}{2}},e_{\frac{m}{2}-1})(e_{\frac{m}{2}})=0$$
and from \eqref{P ei ei =0} it follows that
$$aH(e_{\frac{m}{2}-1})=0$$ 
and hence
$$2ae_{\frac{m}{2}-1}=0.$$
We conclude that $a=0$ and so 
$$P(e_{\frac{m}{2}},e_{\frac{m}{2}})=0.$$
\vspace{0.2cm}

If $m$ is odd and $m\neq 1$, then $\frac{m-1}{2}\ge 1$, $\frac{m+1}{2}<m$ and since $char(k)\neq 3$, we have (by Equation \eqref{relations P 1})
$$P(e_{\frac{m-1}{2}},e_{\frac{m-1}{2}})(e_{\frac{m-1}{2}})=0$$
and hence by \eqref{P ei ei =0 2}
$$bE(e_{\frac{m-1}{2}})=0$$
which means $b=0$ since $Ker(E)={\rm Span}<e_0>$ (cf. Theorem \ref{rep irreds sl(2,k) STRADE}) and $\frac{m-1}{2}\ne 0$. Similarly, since $char(k)\neq 3$, we have (by Equation \eqref{relations P 1})
$$P(e_{\frac{m+1}{2}},e_{\frac{m+1}{2}})(e_{\frac{m+1}{2}})=0$$
and so
$$cF(e_{\frac{m+1}{2}})=0$$
which means $c=0$ since $Ker(F)\subseteq {\rm Span}<e_m>$ (cf. Theorem \ref{rep irreds sl(2,k) STRADE}) and $\frac{m+1}{2}\ne m$. 
\vspace{0.2cm}

To summarise we have now shown that if $m\neq 1$, then 
$$P(e_i,e_i)=0 \quad \forall i\in \llbracket 0,m \rrbracket.$$
\vspace{0.5cm}

Now we suppose $m\neq 1$ and show by induction on $n$ that, for all $n$ in $\llbracket0,m \rrbracket$, 
\begin{equation}
P(e_i,e_{i+n})=0 \quad \forall i \in \llbracket 0,m-n \rrbracket.
\label{reccurence}
\end{equation}
\textit{Base case $(n=0)$: } We have already shown that
$$P(e_i,e_i)=0 \quad \forall i \in \llbracket0,m \rrbracket$$
and so \eqref{reccurence} is true if $n=0$.
\vspace{0.3cm}

\noindent
\textit{Induction: } Suppose that the relation
$$P(e_i,e_{i+k})=0 \quad \forall i \in \llbracket 0,m-k \rrbracket$$
is satisfied for all $k$ in $\llbracket 0,n-1\rrbracket$. We have 
$$[F,P(e_i,e_{i+n-1})]=P(e_{i+1},e_{i+n-1})+P(e_i,e_{i+n}) \quad \forall i \in \llbracket0,m-n \rrbracket$$
but since equation \eqref{reccurence} is satisfied for all $k$ in $\llbracket 0,n-1\rrbracket$ we obtain
$$P(e_i,e_{i+n-1})=0, \quad P(e_{i+1},e_{i+n-1})=0 \quad \forall i \in \llbracket0,m-n \rrbracket$$
and hence 
$$P(e_i,e_{i+n})=0 \quad \forall i \in \llbracket0,m-n \rrbracket.$$
This completes the proof of \eqref{reccurence} by induction and hence the proof of part a) of the Lemma.
\vspace{0.3cm}

Finally, to prove part b) of the Lemma, suppose that $m=1$ and recall that by \eqref{P ei ei =0 2}
$$P(e_0,e_0)=bE, \qquad P(e_1,e_1)=cF.$$
By \eqref{relations P 2}, we have
$$[H,P(e_0,e_1)]=P(H(e_0),e_1)+P(e_0,H(e_1))=P(e_0,e_1)-P(e_0,e_1)=0,$$
and hence there exists $\gamma$ in $k$ such that 
$$P(e_0,e_1)=\gamma H.$$
Using the relation \eqref{relations P 1}, we obtain
$$2P(e_0,e_1)(e_0)+P(e_0,e_0)(e_1)=0$$
which means
$$2\gamma H(e_0)+bE(e_1)=0$$
from which it follows that
$$2\gamma e_0+be_0=0$$
and so $b=-2\gamma$. Similarly, using the relation \eqref{relations P 1} we also have
$$2P(e_0,e_1)(e_1)+P(e_1,e_1)(e_0)=0$$
which means
 $$2\gamma H(e_1)+cF(e_0)=0$$
from which it follows that
$$-2\gamma e_1+ce_1=0$$
and so $c=2\gamma$. Thus, we have
$$P(e_0,e_0)=-2\gamma E, \quad P(e_0,e_1)=\gamma H, \quad P(e_1,e_1)=2\gamma F,$$
and this proves b).
\end{proof}

For the next proposition we do not assume that the odd part of the Lie superalgebra is an irreducible representation of $\sl(2,k)$. However, we show that if it contains a non-trivial irreducible submodule of dimension not three on which $P$ vanishes, then $P$ vanishes identically.

\begin{pp} \label{LM P_VxW=0}
Let $k$ be a field of characteristic zero or an algebraically closed field of positive characteristic. Let $V$ be a finite-dimensional representation of $\sl(2,k)$ and $W\subseteq V$ a non-trivial irreducible submodule. Let $P:V\times V \rightarrow \sl(2,k)$ be a symmetric bilinear map which satisfies the relations \eqref{relations P 2}, \eqref{relations P 1} and $P(W,W)=\lbrace 0 \rbrace$.
\begin{enumerate}[label=\alph*)]
\item If $dim(W)\neq 3$, then 
$$P\equiv 0.$$
\item If $dim(W)=3$, let $\lbrace e_0,e_1,e_2 \rbrace$ be a basis of $W$ and let $\lbrace E,H,F \rbrace$ be an $\sl(2,k)$-triple as in Theorem \ref{rep irreds sl(2,k) STRADE} \ref{rep irreds sl(2,k) STRADE a}. Then there exists $\gamma \in V^*$ such that for all $v$ in $V$,
$$\quad P(v,e_0)=-\gamma(v) E, \quad P(v,e_1)=\gamma(v) H, \quad P(v,e_2)=2\gamma(v) F.$$
\end{enumerate}
\end{pp}

\begin{proof}
Let $v \in V$. Using the relation \eqref{relations P 1}, we obtain
$$2P(v,w)(w)+P(w,w)(v)=0 \quad \forall w \in W$$
which implies that
\begin{equation*}
P(v,w)(w)=0 \quad \forall w \in W.
\end{equation*}
Let $\lbrace e_0,\ldots,e_m \rbrace$ be a basis of $W$ as in Theorem \ref{rep irreds sl(2,k) STRADE} \ref{rep irreds sl(2,k) STRADE a}. Since 
$$P(v,e_i)(e_i)=0 \quad \forall i \in \llbracket 0,m \rrbracket$$
it follows from Remark \ref{ann} that:
\begin{itemize}
\item $\exists a, b \in k, ~~\text{ s.t. }~~ P(v,e_0)=a E, ~ P(v,e_m)=b F$,
\item $P(v,e_i)=0 \quad \forall i \in \llbracket 1,m-1 \rrbracket, ~  i \neq \frac{m}{2}$,
\item if $m$ is even, $\exists c \in k, ~~\text{ s.t. }~~ P(v,e_{\frac{m}{2}})=c H$.
\end{itemize}
By \eqref{relations P 1}, we have
$$P(v,e_0)(e_m)+P(v,e_m)(e_0)=0$$
and hence
\begin{equation} \label{eq a E(e_m)+b F(e_0)=0 preuve LM}
a E(e_m)+b F(e_0)=0.
\end{equation}

Suppose that the representation $W$ is not three-dimensional so that $m-1 \neq 1$. Since $W$ is non-trivial, this implies that $a=b=0$ and hence that
$$P(v,e_0)=P(v,e_m)=0.$$
If $m$ is even, again by \eqref{relations P 1}, we have
$$P(v,e_{\frac{m}{2}})(e_0)+P(v,e_0)(e_{\frac{m}{2}})+P(e_0,e_{\frac{m}{2}})(v)=0$$
and hence $c H(e_0)=0$. Since $H(e_0)\neq 0$, it follows that $c=0$ and $P(v,e_{\frac{m}{2}})=0$.
Therefore, we have
$$P(v,e_i)=0 \quad \forall i \in \llbracket 0,m \rrbracket,$$
and hence by Lemma  \ref{LM general} \ref{LM P_VxW=0 implies P=0}, we have $P\equiv 0$.
\\

Suppose $dim(W)=3$. By \eqref{eq a E(e_m)+b F(e_0)=0 preuve LM} and Theorem \ref{rep irreds sl(2,k) STRADE} \ref{rep irreds sl(2,k) STRADE a} we have
$$2ae_1+be_1=0$$
and then $b=-2a$. By the relation \eqref{relations P 1} we also have
$$P(v,e_0)(e_1)+P(v,e_1)(e_0)=0 \Rightarrow aE(e_1)+cH(e_0)=0 \Rightarrow 2ae_0+2ce_0=0,$$
and so $a=-c$. Thus,
$$P(v,e_0)=a E, \quad P(v,e_1)=-a H, \quad P(v,e_2)=-2a F$$
and $a$ clearly depends linearly on $v$. This proves b).
\end{proof}

\section{Lie superalgebras with three-dimensional simple even part}

In this section we prove the two main theorems of this article. The first is a classification of finite-dimensional Lie superalgebras whose even part is $\sl(2,k)$ under the hypotheses that $k$ is of characteristic zero or an algebraically closed field of positive characteristic. The second extends this classification to the case of finite-dimensional Lie superalgebras whose even part is any three-dimensional simple Lie algebra over an arbitrary field of characteristic not two or three. 

\begin{thm} \label{classification LSA alg clos}
Let $k$ be a field of characteristic zero or an algebraically closed field of positive characteristic. Let $\gg=\sl(2,k) \oplus \gg_1$ be a finite-dimensional Lie superalgebra over $k$ and let 
$$\mathcal{Z}(\gg):= \lbrace x \in \gg ~ | ~ \lbrace x,y \rbrace=0 \quad \forall y \in \gg \rbrace .$$
Then there are three cases:
\begin{enumerate}[label=\alph*)]
\item $\lbrace \gg_1,\gg_1 \rbrace=\lbrace 0 \rbrace$ ;
\item $\gg_1=(\sl(2,k)\oplus k) \oplus \mathcal{Z}(\gg)$, \quad (see Example \ref{ex ss+ss+k}) ;
\item $\gg\cong \osp_k(1|2)\oplus \mathcal{Z}(\gg)$, \quad (see Example \ref{osp(12)}).
\end{enumerate}
\end{thm}

\begin{proof}
Let $\lbrace E,H,F\rbrace$ be an $\sl(2,k)$-triple. We denote by $V$ the representation $\gg_1$ of $\sl(2,k)$. Even if $V$ is not completely reducible, it always has irreducible submodules. To prove the theorem we show the following four implications:
\begin{itemize}
\item V has an irreducible submodule of dimension strictly greater than 3 $\Rightarrow$ $\gg$ as in a) ;
\item V has an irreducible submodule of dimension 2 $\Rightarrow$ $\gg$ as in a) or c) ;
\item V has an irreducible submodule of dimension 3 $\Rightarrow$ $\gg$ as in a) or b) ;
\item V only has irreducible submodules of dimension 1 $\Rightarrow$ $\gg$ as in a).
\end{itemize}

Case 1: If there is an irreducible representation $W\subseteq V$ such that $dim(W)>3$, then by Proposition \ref{1er LM} we have $P|_{W\times W} \equiv 0$. Furthermore, by Proposition \ref{LM P_VxW=0}, we have 
$$P\equiv 0,$$
and $\gg$ is as in a).
\vspace{0.2cm}

Case 2: Let $W\subseteq V$ be an irreducible representation of dimension 2. If $P|_{W \times W}\equiv 0$, then by Proposition \ref{LM P_VxW=0} we have 
$$P\equiv 0$$
and $\gg$ is as in a).
\vspace{0.2cm}

If $P|_{W \times W}\not\equiv 0$, by Lemma \ref{LM general} \ref{non nul implique x(v) dans W}, we have
$$x(v) \in W \quad \forall x \in \sl(2,k), ~ \forall v \in V$$
and, since $W$ is irreducible, this means $V$ has no other non-trivial irreducible submodules. Since
$$E^2|_W=F^2|_W=0,$$
we also have
\begin{equation*}
E^3=F^3=0.
\end{equation*}
Since $char(k)=0$ or $char(k)>3$ then $V$ is completely reducible by Theorem \ref{comp red sl(2)} and so $V= W\oplus V_0$ where $V_0$ is a subspace of $V$ on which $\sl(2,k)$ acts trivially. Let $v$ be in $V_0$. The vector space 
$$I_v:={\rm Span}<P(v,w) \quad \forall w \in W>$$
is at most of dimension 2 and is an ideal of $\sl(2,k)$. Thus $I_v=\lbrace 0 \rbrace$ which implies
$$P(V_0,W)=\lbrace 0 \rbrace$$
since $v \in V_0$ was arbitrary. Furthermore, since $V= W\oplus V_0$ and
$$P(V_0,V_0)=\lbrace 0 \rbrace$$
by Lemma \ref{LM general} \ref{LM structure subrep triviale}, we have
$$P(V_0,V)=\lbrace 0 \rbrace$$
and clearly $V_0=\mathcal{Z}(\gg)$. Let $\lbrace e_0,e_1 \rbrace$ be a basis of $W$ as in Theorem \ref{rep irreds sl(2,k) STRADE}. By Proposition \ref{1er LM} there exists $\gamma \in k^*$ such that
\begin{equation*}
P(e_0,e_0)=-2\gamma E, \quad P(e_0,e_1)=\gamma H, \quad P(e_1,e_1)=2\gamma F
\end{equation*}
and the bracket defined on $W$ is a moment map. Hence $\sl(2,k)\oplus \gg_1 \cong \osp_k(1|2)\oplus \mathcal{Z}(\gg)$ (see Example \ref{osp(12)}) and $\gg$ is as in c).
\vspace{0.4cm}

Case 3: Let $W \subseteq V $ be an irreducible representation of dimension 3. Recall that by Proposition \ref{1er LM}, $P|_{W\times W} \equiv 0$. Let $\lbrace e_0,e_1,e_2 \rbrace$ be a basis of $W$ as in Theorem \ref{rep irreds sl(2,k) STRADE}. By Proposition \ref{LM P_VxW=0}, there exists a $\gamma$ in $V^*$ such that for all $v$ in $V$, we have
$$P(v,e_0)=-\gamma(v) E, \quad P(v,e_1)=\gamma(v) H, \quad P(v,e_2)=2\gamma(v) F.$$
If $\gamma(v)=0$ for all $v$ in $V$, then by Lemma \ref{LM general} \ref{LM P_VxW=0 implies P=0} we have $P\equiv 0$ and $\gg$ is as in a). If there exists $v$ in $V$ such that $\gamma(v) \neq 0$ we proceed as follows.
\vspace{0.2cm}

Let $v'$ in $V$ be such that $\gamma(v') \neq 0$. Then by Relation \eqref{relations P 1} we have
$$2P(v',e_i)(v')=-P(v',v')(e_i) \in W \quad \forall i \in \lbrace 0,1,2 \rbrace,$$
and since $\sl(2,k)={\rm Span}<P(v',e_0),P(v',e_1),P(v',e_2)>$ this implies
\begin{equation}
x(v') \in W \quad \forall x \in \sl(2,k).
\label{eqA}
\end{equation}
Let $v''$ in $V$ be such that $\gamma(v'') = 0$. Then 
$$P(v,e_i)(v'')+P(v'',e_i)(v)+P(v,v'')(e_i)=0 \quad \forall i \in \lbrace 0,1,2 \rbrace$$
implies
$$P(v,e_i)(v'')=-P(v,v'')(e_i) \in W \quad \forall i \in \lbrace 0,1,2 \rbrace$$
and since $\gamma(v) \neq 0$, we have
\begin{equation}
x(v'') \in W \quad \forall x \in \sl(2,k).
\label{eqB}
\end{equation}
From \eqref{eqA} and \eqref{eqB} it follows that
$$x\cdot V \subseteq W \quad \forall x \in \sl(2,k)$$
and hence that
\begin{equation*}
E^4=F^4=0.
\end{equation*}
Since $char(k)=0$ or $char(k)>3$ then $V$ is completely reducible by Theorem \ref{comp red sl(2)} and so, since $W$ is irreducible, $V = W \oplus V_0$ where $V_0$ is a subspace of $V$ on which $\sl(2,k)$ acts trivially. By Lemma \ref{LM general} \ref{LM structure subrep triviale} $P|_{V_0 \times V_0}$ is trivial and hence
$$V_0\cap Ker(\gamma)=\mathcal{Z}(\gg)$$
since $Ker(\gamma)$ is the supercommutant of $W$ in $\gg$. Furthermore, $\gamma : V \rightarrow k$ is a non-trivial linear form vanishing on $W$ so it follows that $V_0\cap Ker(\gamma)=\mathcal{Z}(\gg)$ is of codimension one in $V_0$. Taking $v''' \in V_0$ such that $\gamma(v''')\neq 0$, we obtain
$$\gg_1= (\sl(2,k)\oplus {\rm Span}<v'''>) \oplus \mathcal{Z}(\gg)$$
and $\gg$ is as in b).
\vspace{0.5cm}

Case 4: Now, suppose that the only irreducible submodules of $V$ are trivial. If $char(k)=0$ then $V$ is necessarily trivial and by Lemma \ref{LM general} \ref{LM structure subrep triviale} we have $P\equiv 0$.
\vspace{0.2cm}

We assume that $k$ is algebraically closed and $char(k)>0$. First, suppose that $V$ is indecomposable and non-trivial. Since the only irreducible submodules of $V$ are trivial, we can find a composition series
$$\lbrace 0 \rbrace \subset V_1 \subset \ldots \subset V_n = V$$
where for some $1\leq l\leq n$, $W:=V_l$ is a maximal trivial submodule of $V$. 
\vspace{0.2cm}

\begin{lm}
Let $\gg$ be a simple Lie algebra. Let $M$ be a representation of $\gg$ and let $N \subseteq M$ be a submodule of $M$ such that $N$ and $M/N$ are trivial representations. Then $M$ is a trivial representation.
\label{lm triv max}
\end{lm}
\begin{proof}
Let $y,y' \in \gg$ and $v \in M$. Since $M/N$ is a trivial representation, we have $y'(v) \in N$ and so 
$$y(y'(v))=0,$$
since $N$ is trivial. Let $x \in \gg$. Since $\gg$ is simple, we have $\gg=[\gg,\gg]$ then $x=\Sigma [x_i,x_i']$ where $x_i, x_i' \in \gg$ and hence
$$x(v)=(\Sigma [x_i,x_i'])(v)=\Sigma( x_i(x_i'(v))-x_i'(x_i(v)) )=0.$$
\end{proof}

We recall:
\begin{lm} (Remark 5.3.2 of \cite{Strade})
A composition factor $V_{i+1}/V_i$ is of dimension $1$ or $p-1$.
\end{lm}

\begin{proof}
Since $V$ is indecomposable, the Casimir element $\Omega=(H+Id)^2+4FE$ of $\sl(2,k)$ has a unique eigenvalue. If we compute this eigenvalue on the first composition factor which is trivial, we obtain $1$. Let $\alpha \in k^*$ and let $\lbrace e_0,\ldots,e_m \rbrace$ be a basis of the composition factor $V_{i+1}/V_i$ as in Theorem \ref{rep irreds sl(2,k) STRADE}. We have
$$\Omega(e_0)=(\alpha+1)^2e_0$$
and so $(\alpha+1)^2=1$ which is equivalent to $\alpha(\alpha+2)=0$. Hence $\alpha=[dim(V_{i+1}/V_i)-1]$ by Theorem \ref{rep irreds sl(2,k) STRADE} and so $\alpha(\alpha+2)=0$ is equivalent to $dim(V_{i+1}/V_i)=1$ or $p-1$.
\end{proof}

\begin{cor}\label{remark dim p-1}
The composition factor $V_{l+1}/W$ is of dimension $p-1$.
\end{cor}

\begin{proof}
If $dim(V_{l+1}/W)=1$, then $V_{l+1}$ is trivial by Lemma \ref{lm triv max} which is a contradiction to the maximality of $V_l$. Consequently $V_{l+1}/W$ is of dimension $p-1$.
\end{proof}
\vspace{0.1cm}

Let $w \in W$ and consider
$$I_w:={\rm Span}<P(v,w) ~ | ~ v \in V_{l+1}>,$$
which is an ideal of $\sl(2,k)$. Suppose that $I_w\neq \lbrace 0 \rbrace$, or equivalently that $I_w=\sl(2,k)$.
\vspace{0.2cm}

Since $\sl(2,k)$ acts trivially on $W$ by assumption and since $P(W,W)=\lbrace 0 \rbrace$ by Lemma \ref{LM general} \ref{LM structure subrep triviale}, we have a well-defined equivariant linear map from $V_{l+1}/W$ to $\sl(2,k)$ given by
$$[v] \mapsto P(v,w).$$
This map is surjective by assumption, and injective since $V_{l+1}/W$ is irreducible. Therefore the dimension of $V_{l+1}/W$ is 3, which is impossible since we have seen above that $V_{l+1}/W$ is of dimension $p-1\neq 3$.
\vspace{0.2cm}
 
Consequently, for all $w$ in $W$, $I_w=\lbrace 0 \rbrace$ and then the ideal
$$I:={\rm Span}<P(v,w) ~ | ~ v \in V_{l+1}, ~~ w \in W>,$$
is also trivial. Thus, there is a well-defined symmetric bilinear map $\dot{P} : V_{l+1}/W \times V_{l+1}/W  \mapsto \sl(2,k)$ given by
$$\dot{P}([v_1],[v_2])=P(v_1,v_2)$$
and $\dot{P}$ satisfies the two relations \eqref{relations P 2} and \eqref{relations P 1}. Since $V_{l+1}/W$ is an irreducible representation of $\sl(2,k)$ of dimension $p-1\ge 4$ it follows from Proposition \ref{1er LM} that $\dot{P}\equiv 0$ and so 
$$P|_{V_{l+1} \times V_{l+1}}\equiv 0.$$ 

Let $w \in W$ and $v \in V$. We have
$$P(v,v')(w)+P(v,w)(v')+P(v',w)(v)=0 \quad \forall v' \in V_{l+1}$$
which implies that
$$P(v,w)(v')=0 \quad \forall v' \in V_{l+1}.$$
Since $V_{l+1}$ is a non-trivial representation of $\sl(2,k)$, this means
$$P|_{V \times W}\equiv 0.$$ 
Again, there is a well-defined symmetric bilinear map $\ddot{P} : V/W \times V/W  \mapsto  \sl(2,k)$ given by
$$\ddot{P}([v_1],[v_2])=P(v_1,v_2)$$
and $\ddot{P}$ satisfies the two relations \eqref{relations P 2} and \eqref{relations P 1}. However the representation $V/W$ contains the irreducible representation $V_{l+1}/W$ of dimension $p-1$ and hence, by Proposition \ref{LM P_VxW=0}, we have $\ddot{P}\equiv 0$, and so finally $P\equiv 0$.
\vspace{0.5cm}

Now, if the representation $V$ is decomposable, we have
\begin{equation}
V\cong V_1 \oplus \ldots \oplus V_n
\label{decomp rep sous rep trivial}
\end{equation}
where $V_i$ is indecomposable such that the only irreducible submodules of $V_i$ are trivial. We have just seen that
$$P|_{V_i \times V_i}\equiv 0 \quad \forall i,$$
thus, it remains to prove that $P|_{V_i \times V_j}\equiv 0$ for two indecomposable summands $V_i$ and $V_j$ in the decomposition \eqref{decomp rep sous rep trivial}.
Let
$$\lbrace 0 \rbrace \subset U_1 \subset \ldots \subset U_n = V_i, \quad \lbrace 0 \rbrace \subset U'_1 \subset \ldots \subset U'_m = V_j,$$
be two composition series where for some $1\leq l\leq n$ (resp. $1\leq k\leq m$), $W:=U_l$ (resp. $W':=U'_k$) is a maximal trivial submodule of $V_i$ (resp. $V_j$).
\vspace{0.4cm}

If $V_i$ and $V_j$ are trivial, $P|_{V_i \times V_j}\equiv 0$ By Lemma \ref{LM general} \ref{LM structure subrep triviale}. If $V_i$ is non-trivial, using the same reasoning as above, we will first show that
$$P|_{V_i \times W'}\equiv 0.$$
Let $w \in W'$ and consider the ideal
$$I_w={\rm Span}<P(v,w) ~ | ~ v \in U_{l+1}>.$$
By Corollary \ref{remark dim p-1}, we have $dim(U_{l+1}/W)=p-1$. Suppose that $I_w\neq \lbrace 0 \rbrace$. Since $P(W',W)=\lbrace 0 \rbrace$, we have a well-defined equivariant linear map from $U_{l+1}/W$ to $\sl(2,k)$ defined by
$$[v] \mapsto P(v,w).$$
This map is surjective by assumption and injective since $U_{l+1}/W$ is irreducible. Therefore the dimension of $U_{l+1}/W$ is 3, which is impossible since we have seen that the dimension of  $U_{l+1}/W$ is $p-1$. Thus, for all $w'$ in $W$, $I_w=\lbrace 0 \rbrace$ and then $P(U_{l+1},W')=\lbrace 0 \rbrace$.
\vspace{0.2cm}

Now, let $v\in V$, $w' \in W'$ and $u \in U_{l+1}$. We have
$$P(v,w')(u)+P(u,w')(v)+P(v,u)(w')=0,$$
and since $P(u,w')=0$ and $P(v,u)=0$, this implies
$$P(v,w')(u)=0 \quad \forall u \in U_{l+1}.$$
However $U_{l+1}$ is non-trivial, so $P(v,w')=0$ and hence 
$$P|_{V_i \times W'}\equiv 0.$$
 
If $V_j$ is trivial (which means that $V_j=W'$), this shows that $P|_{V_i \times V_j}\equiv 0$.
\vspace{0.2cm}

If $V_j$ is non-trivial, as we have already seen,
\begin{equation*}
P|_{V_i \times W'}\equiv 0, \quad P|_{V_j \times W}\equiv 0
\end{equation*}
and hence we have a well-defined symmetric bilinear map $\dot{P}$ from $(V_i/W \oplus V_j/W')^2$ to $\sl(2,k)$ given by
$$\dot{P}([v_1],[v_2]):=P(v_1,v_2)$$
and $\dot{P}$ satisfies the two relations \eqref{relations P 2} and \eqref{relations P 1}. By Corollary \ref{remark dim p-1} there is an irreducible submodule of $V_i/W \oplus V_j/W'$ of dimension $p-1$ and then, by Propositions \ref{1er LM} and \ref{LM P_VxW=0}, we have $\dot{P}\equiv 0$. Then $\gg$ is as in a).
\end{proof}
\vspace{0.2cm}

In order to extend our result to fields of positive characteristic which are not necessarily algebraically closed and also to general three-dimensional simple Lie algebras, we need the two following observations.
\begin{pp} \label{non-split 2-dim rep}
Let $\ss$ be a non-split three-dimensional simple Lie algebra. The only two-dimensional representation of $\ss$ is the trivial representation.
\end{pp}

\begin{proof}
Let $\rho : \ss \rightarrow {\rm End(V)}$ be a two-dimensional representation of $\ss$. Since $\ss=[\ss,\ss]$, the representation $\rho$ maps $\ss$ to $\sl(V)$. Since $\ss$ is simple, the representation is either trivial or an isomorphism but since $\ss \not\cong \sl(2,k)$, the representation is trivial.
\end{proof}
\vspace{0.2cm}

\begin{lm} \label{supercentre}
Let $\gg$ be a  Lie superalgebra over $k$ and let $\kt/k$ be an extension. We have
$$\mathcal{Z}(\gg\otimes \kt)=\mathcal{Z}(\gg)\otimes \kt.$$
\end{lm}

\begin{proof}
The inclusion $\mathcal{Z}(\gg)\otimes \kt \subseteq \mathcal{Z}(\gg\otimes \kt)$ is clear.
\vspace{0.2cm}

Let $x \in \mathcal{Z}(\gg\otimes \kt)$ and let $\lbrace e_i \rbrace_{i \in I}$ be a $k$-basis of $\kt$. Then there exist $\lbrace v_i \rbrace_{i \in I}$ such that $v_i \in \gg$ and $x=\sum\limits_{i \in I} v_i \otimes e_i$. For all $y$ in $\gg$ we have
$$\lbrace x,y\otimes1\rbrace=0$$
which implies
$$\sum_{i \in I} \lbrace v_i,y \rbrace \otimes e_i=0$$
and hence we have
$$\lbrace v_i,y\rbrace =0 \quad \forall i \in I, ~ \forall y \in \gg.$$
Consequently $x \in \mathcal{Z}(\gg)\otimes \kt$.
\end{proof}
\vspace{0.2cm}

We can now prove the most important result of the paper.

\begin{thm} \label{thm final}
Let $\gg=\gg_0 \oplus \gg_1$ be a finite-dimensional Lie superalgebra such that $\gg_0$ is a three-dimensional simple Lie algebra and let
$$\mathcal{Z}(\gg):= \lbrace x \in \gg ~ | ~ \lbrace x,y \rbrace=0, ~ \forall y \in \gg \rbrace .$$
Then, there are three cases:
\begin{enumerate}[label=\alph*)]
\item $\lbrace \gg_1,\gg_1 \rbrace=\lbrace 0 \rbrace$ ;
\item $\gg_1= (\gg_0\oplus k) \oplus \mathcal{Z}(\gg)$, \quad (see Example \ref{ex ss+ss+k}) ;
\item $\gg\cong \osp_k(1|2)\oplus \mathcal{Z}(\gg)$, \quad (see Example \ref{osp(12)}).
\end{enumerate}
\end{thm}

\begin{proof}
Let $\bar{k}$ be the algebraic closure of $k$ and set $\bar{\gg}:=\gg\otimes \bar{k}$, $\bar{\gg_0}:=\gg_0\otimes \bar{k}$ and $\bar{\gg_1}:=\gg_1\otimes \bar{k}$. Since $\bar{\gg_0}\cong \sl(2,\bar{k})$ it follows from Theorem \ref{classification LSA alg clos} that $\bar{\gg}$ satisfies one of the following:
\begin{enumerate}[label=\alph*)]
\item $\lbrace \bar{\gg_1},\bar{\gg_1} \rbrace=\lbrace 0 \rbrace$ ;
\item $\bar{\gg}_1\cong(\sl(2,\bar{k})\oplus \bar{k}) \oplus \mathcal{Z}(\bar{\gg})$ ;
\item $\bar{\gg}\cong \osp_{\bar{k}}(1|2)\oplus \mathcal{Z}(\bar{\gg})$.
\end{enumerate}
\vspace{0.2cm}

Case $a)$: If $\lbrace \bar{\gg_1},\bar{\gg_1}\rbrace=\lbrace 0 \rbrace$ then clearly $\lbrace \gg_1,\gg_1 \rbrace = \lbrace 0 \rbrace$.
\vspace{0.2cm}

Case $b)$: Suppose that $\bar{\gg_1} \cong \bar{V} \oplus \mathcal{Z}(\bar{\gg})$ where $\bar{V}$ is the direct sum of the adjoint representation and a one-dimensional trivial representation of $\bar{\gg_0}\cong\sl(2,\bar{k})$. We recall Proposition 3.13 of \cite{Bourbaki1}:

\begin{pp} \label{bourbaki}
Let $\hh$ be a Lie algebra, let $V$ and $W$ be two finite-dimensional representations of $\hh$ and let $\kt/k$ be an extension field. If $V\otimes \kt$ and $W\otimes \kt$ are isomorphic as representations of $\hh\otimes\kt$, then $V$ and $W$ are isomorphic as representations of $\hh$.
\end{pp}

From this, it follows that there is a direct sum decomposition
$$\gg_1=V_1\oplus V_0$$
where $\gg_0$ acts by the adjoint representation on $V_1$ and trivially on $V_0$. By Lemma \ref{LM general} \ref{LM structure subrep triviale}, $P$ restricted to $V_0$ vanishes identically. If $P|_{V_1\times V_1}\not\equiv \lbrace 0 \rbrace$ then $\bar{\gg_0} \oplus V_1 \otimes \bar{k}$ would be a counter-example to Theorem \ref{classification LSA alg clos} so we deduce that $P|_{V_1\times V_1}\equiv \lbrace 0 \rbrace$.
\vspace{0.2cm}

Since $\mathcal{Z}(\gg)\subseteq V_0$ and since $V_0\otimes \bar{k}$ is characterised as the subspace of $\bar{\gg_1}$ on which $\bar{\gg_0}$ acts trivially, it follows from Lemma \ref{supercentre} that $\mathcal{Z}(\gg)$ is of codimension one in $V_0$. Hence there exists $v\in V_0$ such that
$$V_0={\rm Span}<v>\oplus \mathcal{Z}(\gg)$$
and we have
$$\gg=\gg_0\oplus \left(V_1\oplus {\rm Span}<v> \oplus \mathcal{Z}(\gg)\right).$$
Since $P|_{V_1\times V_1}\equiv \lbrace 0 \rbrace$ and $P(v,v)=0$ we must have $P(v,V_1)\not\equiv\lbrace 0 \rbrace$. However the map $v_1 \mapsto P(v,v_1)$ is a $\gg_0$-equivariant isomorphism from $V_1$ to $\gg_0$ and hence is uniquely determined up to a constant (see Example \ref{ex ss+ss+k}). It is now easy to check that this implies $b)$.
\vspace{0.2cm}

Case $c)$: Now, suppose that $\gg_1 \otimes \bar{k} \cong \bar{V} \oplus \mathcal{Z}(\gg \otimes \bar{k})$ where $\bar{V}$ is the standard representation of $\gg_0\otimes \bar{k}\cong \sl(2,\bar{k})$. By Lemma \ref{supercentre}, we have $\mathcal{Z}(\gg\otimes \bar{k})=\mathcal{Z}(\gg)\otimes \bar{k}$ and so $\gg_1/\mathcal{Z}(\gg)$ is an irreducible two-dimensional representation of $\gg_0$ which implies $\gg_0\cong \sl(2,k)$ by Proposition \ref{non-split 2-dim rep}. By Proposition \ref{bourbaki}, there is a direct sum decomposition
$$\gg_1=V_1\oplus \mathcal{Z}(\gg)$$
where $\gg_0\cong \sl(2,k)$ acts on $V_1$ by the standard representation. By Lemmas \ref{LM general} and \ref{1er LM} \ref{1er LM b} the bracket on $V_1$ is a moment map and so $\gg\cong \osp_k(1|2)\oplus \mathcal{Z}(\gg)$.
\end{proof}
\vspace{0.2cm}

Over a field of positive characteristic, there is the important notion of \textit{restricted} Lie superalgebra (\cite{PETROGRADSKI19921} or \cite{Wang2000}) and then we obtain the following result.
\vspace{0.2cm}

\begin{cor} \label{cor restricted}
Let $k$ be a field of positive characteristic $p$, let $\gg=\gg_0 \oplus \gg_1$ be a finite-dimensional Lie superalgebra over $k$ such that $\gg_0$ is a three-dimensional simple Lie algebra, such that $\lbrace \gg_1,\gg_1 \rbrace \neq \lbrace 0\rbrace$ and let $K$ be the Killing form of $\gg_0$.
\begin{enumerate}[label=\alph*)]
\item The Lie superalgebra $\gg$ is restricted.
\item The $[p]$-map on $\gg_0$ satisfies
\begin{equation} \label{p map}
x^{[p]}=\Big( \frac{K(x,x)}{2} \Big)^{\frac{p-1}{2}}x \qquad \forall x \in \gg_0.
\end{equation}
\item If $\gg_1=(\gg_0\oplus k) \oplus \mathcal{Z}(\gg)$ the $[2p]$-map on $\gg_1$ satisfies
\begin{equation*}
(x+\lambda)^{[2p]}=\lambda^p\Big( \frac{K(x,x)}{2} \Big)^{\frac{p-1}{2}}x \qquad \forall x+\lambda \in \gg_0\oplus k
\end{equation*}
where the right-hand side of this equation is to be understood as an even element of $\gg$.
\item If $\gg\cong\osp_k(1|2)\oplus \mathcal{Z}(\gg)$ the $[2p]$-map on $\gg_1$ is trivial.
\end{enumerate}
\end{cor}
\vspace{0.1cm}

\begin{proof}
\begin{enumerate}[label=\alph*)]
\item The Killing form of $\gg_0$ is non-degenerate (see \cite{MALCOLMSON1992}) and hence $\gg_0$ is a restricted Lie algebra (see page 191 of \cite{jacobson1979lie}). This means by definition that the adjoint representation of $\gg_0$ is a restricted representation. It is well-known that the trivial representation of $\gg_0$ and the standard representation $k^2$ of $\sl(2,k)$ are also restricted representations. Hence, in the sense of V. Petrogradski (\cite{PETROGRADSKI19921}), the Lie superalgebras $\gg_0\oplus(\gg_0\oplus k) \oplus \mathcal{Z}(\gg)$ and $\osp_k(1|2)\oplus \mathcal{Z}(\gg)$ appearing in Theorem \ref{thm final} are restricted Lie superalgebras.
\item Let $x\in \gg_0$ be non-zero and let $y,z\in \gg_0$ be such that $\lbrace x,y,z \rbrace$ is a basis of $\gg_0$. The matrix of $ad(x)$ in the basis $\lbrace x,y,z \rbrace$ is
$$\begin{pmatrix} 
0 & a & d \\
0 & b & e \\
0 & c & f 
\end{pmatrix},$$
where $a,b,c,d,e,f \in k$. Since $\gg_0$ is simple, $[\gg_0,\gg_0]=\gg_0$ and so $Tr(ad(x))=0$. Thus $b=-f$ and the characteristic polynomial $P_x(X)$ of $ad(x)$ is 
$$P_x(X)=-X\Big(X^2+(-b^2-ec)\Big).$$
One can check that 
$$K(x,x)=2(b^2+ec)$$
and then
$$P_x(X)=-X\Big(X^2-\frac{K(x,x)}{2}\Big).$$
By the Cayley–Hamilton theorem we have
$$ad(x)^3=\frac{K(x,x)}{2}ad(x)$$
and so
$$ad(x)^p=\Big(\frac{K(x,x)}{2}\Big)^{\frac{p-1}{2}}ad(x).$$
\item By \cite{BLLS14} we have
\begin{equation} \label{2p map}
x^{[2p]}=\Big( \frac{1}{2}P(x,x) \Big)^{[p]} \qquad \forall x \in \gg_1.
\end{equation}
If $\gg_1=(\gg_0\oplus k) \oplus \mathcal{Z}(\gg)$, since $(y+z)^{[2p]}=y^{[2p]}$ for all $y \in \gg_1$, for all $z\in \mathcal{Z}(\gg)$, we will only consider $y\in \gg_1$ of the form $y=x+\lambda$ where $x\in \gg_0$ and $\lambda \in k$. Using Equation \eqref{2p map} and Definition \ref{def dbl} we have
\begin{align*}
(x+\lambda)^{[2p]}&=\Big( \frac{1}{2}P(x+\lambda,x+\lambda) \Big)^{[p]}\\
&=P(x,\lambda)^{[p]}\\
&=(\lambda x)^{[p]}
\end{align*}
and using Equation \eqref{p map} we obtain
\begin{align*}
(x+\lambda)^{[2p]}&=\Big( \frac{K(\lambda x,\lambda x)}{2} \Big)^{\frac{p-1}{2}}\lambda x\\
&=\lambda^p\Big( \frac{K(x,x)}{2} \Big)^{\frac{p-1}{2}} x.
\end{align*}
\item Suppose that $\gg\cong\osp_k(1|2)\oplus \mathcal{Z}(\gg)$ and let $(a,b) \in k^2$. Using Equation \eqref{2p map} and Example \ref{osp(12)}  we have
\begin{align*}
(a,b)^{[2p]}&=\Big( \frac{1}{2}P((a,b),(a,b)) \Big)^{[p]}\\
&=\begin{pmatrix}
   -ab & a^2 \\
   -b^2 & ab 
\end{pmatrix}^{[p]}.
\end{align*}
Using Equation \eqref{p map} and the well-known fact that
$$K(x,x)=-8\cdot det(x) \qquad \forall x \in \sl(2,k)$$
we obtain
\begin{align*}
(a,b)^{[2p]}&=\Big( -4\cdot det\begin{pmatrix}
   -ab & a^2 \\
   -b^2 & ab 
\end{pmatrix} \Big)^{\frac{p-1}{2}}\begin{pmatrix}
   -ab & a^2 \\
   -b^2 & ab 
\end{pmatrix}\\
&=0.
\end{align*}
\end{enumerate}
\end{proof}
\vspace{0.1cm}

Throughout the paper, we have always assumed the base field $k$ to be of characteristic not two or three. Here are some comments on this assumption.

\begin{itemize}
\item If $k$ is of characteristic three, the definition of a Lie superalgebra $\gg=\gg_0 \oplus \gg_1$ is usually modified by adding the property
\begin{equation} \label{axiom def LSA char 3}
\lbrace x, \lbrace x , x \rbrace \rbrace =0 \qquad \forall x \in \gg_1
\end{equation}
to those of Definition \ref{def LSA}. We will give a counter-example to Theorem \ref{classification LSA alg clos} with this definition of a Lie superalgebra in characteristic three.

\item If $k$ is of characteristic two, the definition of a Lie superalgebra is also usually modified, see \cite{LEB}. In this characteristic there are many reasons for which our proof totally fails. The most important of these, is that $\sl(2,k)$ is nilpotent, not simple. On the other hand, an analogue of $\osp_k(1|2)$ can be defined (see Remark 2.2.1 in \cite{BLLS14}) and the Lie superalgebra of Definition \ref{def dbl} can also be defined. In characteristic two, there are counter-examples to Theorem \ref{classification LSA alg clos} (see below). However, there are still  three-dimensional simple Lie algebras and to the best of our knowledge, there is no counter-example to Theorem \ref{thm final} (with $\osp_k(1|2)\oplus \mathcal{Z}(\gg)$ removed from the statement).
\end{itemize}
\vspace{0.1cm}

\begin{ex} \label{ex char 3}
Suppose that $char(k)=3$ and let $\lbrace E,H,F \rbrace$ be an $\sl(2,k)$-triple.
\vspace{0.2cm}

Let $V$ be a three-dimensional vector space with a basis $\lbrace e_0,e_1,v \rbrace$. We define a representation $\rho : \sl(2,k) \rightarrow {\rm End(V)}$ by:
$$\rho(H)=\begin{pmatrix}
1&0&0 \\
0&-1&0 \\
0&0&0
\end{pmatrix}, ~ \rho(E)=\begin{pmatrix}
0&1&0 \\
0&0&1 \\
0&0&0
\end{pmatrix}, ~ \rho(F)=\begin{pmatrix}
0&0&1 \\
1&0&0 \\
0&0&0
\end{pmatrix}.$$
We define a symmetric bilinear map $P : V\times V \rightarrow \sl(2,k)$ by
\begin{alignat*}{3}
P(e_0,e_0)&=E, \quad P(e_0,e_1)&&=H, \quad P(e_1,e_1)&&=-F,\\
P(v,e_0)&=F, \quad P(v,v)&&=H, \quad P(v,e_1)&&=-E,
\end{alignat*}
which satisfies the identities \eqref{relations P 2}, \eqref{relations P 1} and \eqref{axiom def LSA char 3} and hence we obtain a structure of Lie superalgebra on the vector space $\sl(2,k)\oplus V$.
\vspace{0.2cm}

In fact, the linear subspace $\sl(2,k)\oplus {\rm Span}<e_0,e_1>$ is a Lie subsuperalgebra isomorphic to $\osp_k(1|2)$.
\end{ex}
\vspace{0.1cm}

\begin{ex} \label{ex char 2}
Suppose that $char(k)=2$.
\vspace{0.2cm}

Let $\gg_1$ be a Lie algebra isomorphic to $\sl(2,k)$ and let $\phi : \sl(2,k) \rightarrow \gg_1$ be an isomorphism of Lie algebras. Let $\lbrace E,H,F \rbrace$ be an $\sl(2,k)$-triple, i.e., 
$$[E,F]=H, \quad [H,E]=0, \quad [H,F]=0.$$
The Lie algebra $\sl(2,k)$ acts on $\gg_1$ by the adjoint representation:
$$y(x):=[\phi(y),x] \qquad \forall x\in \gg_1, ~ \forall y\in \sl(2,k).$$
If $x=a\phi(E)+b\phi(H)+c\phi(F)\in \gg_1$, we define:
$$x^2:=acH.$$
Let $y=a'\phi(E)+b'\phi(H)+c'\phi(F)\in \gg_1$ and define
$$P(x,y):=(x+y)^2+x^2+y^2=(ac'+a'c)H.$$
We now show that the super vector space $\gg:=\sl(2,k)\oplus \gg_1$ together with the bracket defined by the Lie bracket on $\sl(2,k)$, the adjoint representation $\sl(2,k) \rightarrow {\rm End(\gg_1)}$ and the bilinear symmetric map $P(\phantom{a},\phantom{a})$ is a Lie superalgebra. According to \cite{LEB} this is equivalent to
$$\lbrace x^2,y\rbrace=\lbrace x, \lbrace x,y\rbrace \rbrace \qquad \forall x \in \gg_1, ~ \forall y \in \gg$$
and it is easy to see that both sides of this equation vanish for all $x \in \gg_1, ~ y \in \gg$.
\end{ex}

\section{Acknowledgements}

The author wants to express his gratitude to Marcus J. Slupinski for his encouragement and his expert advice. The author would also like to thank the referee for his careful reading and his suggestions.

\footnotesize

\bibliographystyle{alpha}
\bibliography{articleLSA3DSLA}

\end{document}